\documentclass[10pt]{article}
\def\version{August 15, 2011}

\nonstopmode

\usepackage[dvips]{color}
\usepackage[dvips]{epsfig}
\usepackage{latexsym}
\usepackage{bm}
\usepackage{upgreek}

\usepackage{mathrsfs}
\usepackage{times}
\usepackage{amsthm}
\usepackage{amssymb}
\usepackage{epsfig}
\usepackage{graphicx}
\usepackage{amsmath}

\DeclareOldFontCommand{\brianup}{\upshape}{\mathrm}

\DeclareSymbolFont{EUR}{U}{eur}{m}{n}
\SetSymbolFont{EUR}{bold}{U}{eur}{b}{n}
\DeclareSymbolFontAlphabet{\eur}{EUR}

\DeclareSymbolFont{EUB}{U}{eur}{b}{n}
\SetSymbolFont{EUB}{bold}{U}{eur}{b}{n}
\DeclareSymbolFontAlphabet{\eub}{EUB}

\DeclareSymbolFont{AMSb}{U}{msb}{m}{n}
\DeclareSymbolFontAlphabet{\mathbb}{AMSb}

\newcommand{\notyet}[1]{{}}

\newcommand{\p}{\partial}
\newcommand{\at}[1]{\vert\sb{\sb{#1}}}

\def\R{\mathbb{R}}
\newcommand{\C}{\mathbb{C}}

\newcommand{\N}{\mathbb{N}}

\newcommand{\abs}[1]{\vert #1 \vert}

\newcommand{\norm}[1]{\Vert #1 \Vert}

\newcommand{\sothat}{{\rm ;}\ }

\newcommand{\Range}{\mathop{\rm Range}}

\DeclareMathSymbol{\varGamma}{\mathord}{letters}{"00}
\DeclareMathSymbol{\varDelta}{\mathord}{letters}{"01}
\DeclareMathSymbol{\varTheta}{\mathord}{letters}{"02}
\DeclareMathSymbol{\varLambda}{\mathord}{letters}{"03}
\DeclareMathSymbol{\varXi}{\mathord}{letters}{"04}
\DeclareMathSymbol{\varPi}{\mathord}{letters}{"05}
\DeclareMathSymbol{\varSigma}{\mathord}{letters}{"06}
\DeclareMathSymbol{\varUpsilon}{\mathord}{letters}{"07}
\DeclareMathSymbol{\varPhi}{\mathord}{letters}{"08}
\DeclareMathSymbol{\varPsi}{\mathord}{letters}{"09}
\DeclareMathSymbol{\varOmega}{\mathord}{letters}{"0A}

\theoremstyle{plain}

\newtheorem{lemma}{Lemma}[section]
\newtheorem{corollary}[lemma]{Corollary}

\theoremstyle{definition}
\newtheorem{definition}[lemma]{Definition}

\theoremstyle{remark}
\newtheorem{remark}[lemma]{Remark}

\makeatletter\@addtoreset{equation}{section}
\makeatletter\@addtoreset{theorem}{section}
\makeatother

\def\spec{\sigma}

\renewcommand{\Re}{\mathop{\rm{R\hskip -1pt e}}\nolimits}
\renewcommand{\Im}{\mathop{\rm{I\hskip -1pt m}}\nolimits}

\begin{document}

\title{On the meaning of the Vakhitov-Kolokolov stability criterion
for the nonlinear Dirac equation}

\author{
{\sc Andrew Comech}
\\
{\it\small Texas A\&M University, College Station, Texas, U.S.A.}
{\small and}
{\it\small IITP,
Moscow, Russia}
}


\date{\version}

\maketitle

\begin{abstract}
We consider the spectral stability
of solitary wave solutions
$\phi(x)e^{-i\omega t}$
to the nonlinear Dirac equation
in any dimension.
This equation is well-known to theoretical physicists
as the Soler model
(or, in one dimension, the Gross-Neveu model),
and attracted much attention for many years.
We show that,
generically,
at the values of $\omega$
where the Vakhitov-Kolokolov stability criterion
breaks down,
a pair of real eigenvalues (one positive, one negative)
appears from the origin,
leading to the linear instability of corresponding solitary waves.

As an auxiliary result,
we state the Virial identities
(``Pohozhaev theorem'')
for the nonlinear Dirac equation.
We also show that $\pm 2\omega i$
are the eigenvalues of the nonlinear Dirac equation
linearized at $\phi(x)e^{-i\omega t}$,
which are embedded into the essential spectrum
as long as $\abs{\omega}>m/3$.
This result holds for the nonlinear Dirac equation
with any nonlinearity
of the Soler form
(``scalar-scalar interaction'')
and in any dimension.

As an illustration of the spectral stability methods,
we revisit Derrick's theorem
and sketch
the Vakhitov-Kolokolov stability criterion for the
nonlinear Schr\"odinger equation.
\end{abstract}

\section{Introduction}

Field equations
with nonlinearities
of local type
are natural candidates
for developing tools
which are then used
for the analysis of systems
of interacting equations.
Equations with local nonlinearities
have been appearing in the Quantum Field Theory
starting perhaps since fifties
\cite{PhysRev.84.1,PhysRev.84.10},
in the context
of the classical nonlinear meson theory of nuclear forces.
The nonlinear version of the Dirac equation
is known as the Soler model \cite{PhysRevD.1.2766}.
The existence of standing waves in this model
was proved in
\cite{PhysRevD.1.2766,MR847126}.
Existence of localized solutions to the Dirac-Maxwell system
was addressed in \cite{wakano-1966,MR1364144}
and finally was proved in \cite{MR1386737} (for $\omega\in(-m,0)$)
and \cite{MR1618672} (for $\omega\in(-m,m)$).
The local well-posedness
of the Dirac-Maxwell system
was considered in \cite{MR1391520}.
The local and global well-posedness
of the Dirac equation
was further addressed in
\cite{MR1434039}
(semilinear Dirac equation in $n=3$),
\cite{MR1760283}
(Dirac -- Klein-Gordon system in $n=1$),
and in \cite{MR2108356}
(nonlinear Dirac equation in $n=3$).
The question of stability of solitary wave solutions
to the nonlinear Dirac equation
attracted much attention for many years,
but only partial numerical results were obtained;
see e.g.
\cite{MR637021,MR710348,as1983,as1986,chugunova-thesis}.
The analysis of stability with respect to dilations
is performed in \cite{MR848095,PhysRevE.82.036604}.

Understanding the linear stability is the first step
in the study of stability properties of solitary waves.
Absence of an eigenvalue
with a positive real part
will be referred to as the
spectral stability,
while its absence
as the spectral (or linear) instability.
After the spectrum of the linearized problem
for the nonlinear Schr\"odinger equation~\cite{VaKo}
was understood,
the linearly unstable solitary waves
can be proved to be
( ``nonlinearly'', or ``dynamically'') unstable
\cite{MR948770,2010arXiv1009.5184G},
while
the linearly stable solitary waves
of the nonlinear Schr\"odinger and Klein-Gordon equations
\cite{MR723756,MR804458,MR820338}
and more general $\mathbf{U}(1)$-invariant systems~\cite{MR901236}
were proved to be orbitally stable.
The tools used to prove orbital stability
break down for the Dirac equation
since the corresponding energy functional is sign-indefinite.
On the other hand,
one can hope to use the dispersive estimates
for the linearized equation
to prove the asymptotic stability of the standing waves,
similarly to how it is being done
for the nonlinear Schr\"odinger equation
\cite{MR783974},
\cite{MR1170476},
\cite{MR1199635e},
\cite{MR1681113}, and~\cite{MR1835384}.
The first results on asymptotic stability
for the nonlinear Dirac equation
are already appearing~\cite{2010arXiv1008.4514P,2011arXiv1103.4452B},
with
the assumptions on the spectrum of the linearized equation
playing a crucial role.


In this paper,
we study the spectrum of the nonlinear Dirac equation
linearized at a solitary wave,
concentrating
on bifurcation of real eigenvalues
from $\lambda=0$.

\subsubsection*{Derrick's theorem}
As a warm-up, let us consider
the linear instability of stationary solutions
to a nonlinear wave equation,
\begin{equation}\label{dvk-nlwe}
-\ddot\psi=-\Delta\psi+g(\psi),
\qquad
\psi=\psi(x,t)\in\R,
\quad
x\in\R^n,
\quad
n\ge 1.
\end{equation}
We assume that the nonlinearity $g(s)$ is smooth.
Equation \eqref{dvk-nlwe}
is a Hamiltonian system,
with the Hamiltonian
$E(\psi,\pi)
=\int\sb{\R^n}
\Big(
\frac{\pi^2}{2}
+
\frac{\abs{\nabla\psi}^2}{2}
+G(\psi)
\Big)\,dx,
$
where $G(s)=\int\sb 0\sp s g(s')\,ds'$.

There is a well-known result \cite{MR0174304}
about non-existence of stable localized stationary solutions
in dimension $n\ge 3$
(known as \emph{Derrick's Theorem}).
If $u(x,t)=\theta(x)$ is a localized stationary solution
to the Hamiltonian equations $\dot\pi=-\delta\sb{\psi}E$,
$\dot\psi=\delta\sb{\pi}E$,
then, considering the family $\theta\sb\lambda(x)=\theta(\lambda x)$,
one has
$
\p\sb\lambda\at{\lambda=1}E(\phi\sb\lambda)=0,
$
and then it follows that
$
\p\sb\lambda^2\at{\lambda=1}E(\phi\sb\lambda)<0
$
as long as $n\ge 3$.
That is, $\delta^2 E<0$ for a variation corresponding to
the uniform stretching,
and the solution $\theta(x)$ is to be unstable.
Let us modify Derrick's argument
to show the linear instability of stationary solutions in any dimension.

\begin{lemma}[Derrick's theorem for $n\ge 1$]
For any $n\ge 1$,
a smooth
finite energy stationary solution
$\theta(x)$
to the nonlinear wave equation
is linearly unstable.
\end{lemma}

\begin{proof}
Since $\theta$ satisfies
$-\Delta\theta+g(\theta)=0$,
we also have $-\Delta\p\sb{x\sb 1}\theta+g'(\theta)\p\sb{x\sb 1}\theta=0$.
Due to $\lim\limits\sb{\abs{x}\to\infty}\theta(x)=0$,
$\p\sb{x\sb 1}\theta$ vanishes somewhere.
According to the minimum principle,
there is a nowhere vanishing smooth function
$\chi\in H\sp\infty(\R^n)$
(due to $\Delta$ being elliptic)
which corresponds to some smaller (hence negative) eigenvalue
of 
$\eur{L}=-\Delta+g'(\theta)$,
$\eur{L}\chi=-c^2\chi$, with $c>0$.
Taking $\psi(x,t)=\theta(x)+r(x,t)$,
we obtain the linearization at $\theta$,
$-\ddot r=-\eur{L} r$,
which we rewrite as
$
\p\sb t
\left[\begin{matrix}r\\s\end{matrix}\right]
=
\left[\begin{matrix}0&1\\-\eur{L}&0\end{matrix}\right]
\left[\begin{matrix}r\\s\end{matrix}\right].
$
The matrix in the right-hand side has
eigenvectors
$\left[\begin{matrix}\chi\\\pm c\chi\end{matrix}\right]$,
corresponding to the eigenvalues
$\pm c\in\R$; thus, the solution $\theta$
is linearly unstable.

Let us also mention that
$\p\sb\tau^2\at{\tau=0}E(\theta+\tau\chi)<0$,
showing that $\delta^2 E(\theta)$
is not positive-definite.
\end{proof}

\begin{remark}
A more general result
on the linear stability
and (nonlinear) instability
of stationary solutions
to \eqref{dvk-nlwe}
is in \cite{MR2356215}.
In particular, it is shown there that
the linearization at a stationary solution
may be spectrally stable
when this particular stationary solution is not from $H\sp 1$
(such examples exist in higher dimensions).
\end{remark}

\subsubsection*{Vakhitov-Kolokolov stability criterion
for the nonlinear Schr\"odinger equation}

To get a hold of stable localized solutions,
Derrick suggested that
{\it elementary particles might correspond to stable, 
localized solutions which are periodic in time,
rather than time-independent.}
Let us consider how this works for the (generalized)
nonlinear Schr\"odinger equation
in one dimension,
\begin{equation}\label{dvk-nls}
i\p\sb t\psi=-\frac 1 2\p\sb x^2\psi+g(\abs{\psi}^2)\psi,
\qquad
\psi=\psi(x,t)\in\C,
\quad
x\in\R,
\quad
t\in\R,
\end{equation}
where $g(s)$ is a smooth function
with $m:=g(0)>0$.
One can easily construct solitary wave solutions
$\phi(x)e^{-i\omega t}$,
for some $\omega\in\R$ and $\phi\in H^1(\R)$:
$\phi(x)$ satisfies the stationary equation
$\omega\phi=-\frac 1 2\phi''+g(\phi^2)\phi$,
and can be chosen strictly positive,
even, and monotonically decaying away from $x=0$.
The value of $\omega$ can not exceed $m$.
We consider the Ansatz
$\psi(x,t)=(\phi(x)+\rho(x,t))e^{-i\omega t}$,
with $\rho(x,t)\in\C$.
The linearized equation on $\rho$
is called the linearization at a solitary wave:
\begin{equation}\label{dvk-nls-lin}
\p\sb t
\eub{R}
=
\eub{J}\eub{L}\eub{R},
\qquad
\eub{R}(x,t)=\left[\begin{matrix}\Re\rho(x,t)\\\Im\rho(x,t)\end{matrix}\right],
\end{equation}
with
\begin{equation}\label{dvk-def-l0-l1}
\eub{J}=\left[\begin{matrix}0&1\\-1&0\end{matrix}\right],
\qquad
\eub{L}=\left[\begin{matrix}\eur{L}\sb{+}&0\\0&\eur{L}\sb{-}\end{matrix}\right],
\qquad
\eur{L}\sb{-}=-\frac 1 2\p\sb x^2+g(\phi^2)-\omega,
\qquad
\eur{L}\sb{+}=\eur{L}\sb{-}+2 g'(\phi^2)\phi^2.
\end{equation}
Note that
since $\eur{L}\sb{-}\ne \eur{L}\sb{+}$,
the action of $\eub{L}$
on $\rho$ considered as taking values in $\C$
is $\R$-linear but not $\C$-linear.
Since $\lim\limits\sb{\abs{x}\to\infty}\phi(x)=0$,
the essential spectrum of $\eur{L}\sb{-}$ and $\eur{L}\sb{+}$
is $[m-\omega,+\infty)$.

First, let us note
that the spectrum of $\eub{J}\eub{L}$ is located
on the real and imaginary axes only:
$\sigma(\eub{J}\eub{L})\subset\R\cup i\R$.
To prove this, we consider
$
(\eub{J}\eub{L})^2
=
-\left[\begin{matrix}
\eur{L}\sb{-}\eur{L}\sb{+}&0\\
0&\eur{L}\sb{+}\eur{L}\sb{-}
\end{matrix}\right].
$
Since $\eur{L}\sb{-}$ is positive-definite
($\phi\in\ker\eur{L}\sb{-}$,
being nowhere zero,
corresponds to its smallest eigenvalue),
we can define the selfadjoint root of $\eur{L}\sb{-}$;
then
\[
\sigma\sb d((\eub{J}\eub{L})^2)
\backslash\{0\}
=\sigma\sb d(\eur{L}\sb{-}\eur{L}\sb{+})
\backslash\{0\}
=\sigma\sb d(\eur{L}\sb{+}\eur{L}\sb{-})
\backslash\{0\}
=\sigma\sb d(\eur{L}\sb{-}^{1/2}\eur{L}\sb{+}\eur{L}\sb{-}^{1/2})
\backslash\{0\}
\subset\R,
\]
with the inclusion
due to
$\eur{L}\sb{-}^{1/2}\eur{L}\sb{+}\eur{L}\sb{-}^{1/2}$
being selfadjoint.
Thus,
any eigenvalue $\lambda\in\sigma\sb{d}(\eub{J}\eub{L})$
satisfies
$\lambda^2\in\R$.

Given the family of solitary waves,
$\phi\sb\omega(x)e^{-i\omega t}$,
$\omega\in\Omega\subset\R$,
we would like to know
at which $\omega$
the eigenvalues
of the linearized equation
with $\Re\lambda>0$ appear.
Since $\lambda^2\in\R$,
such eigenvalues can only be located
on the real axis,
having bifurcated from $\lambda=0$.
One can check that
$\lambda=0$ belongs to the discrete spectrum of $\eub{J}\eub{L}$,
with
\[
\eub{J}\eub{L}\left[\begin{matrix}0\\\phi\sb\omega\end{matrix}\right]=0,
\qquad
\eub{J}\eub{L}\left[\begin{matrix}-\p\sb\omega\phi\sb\omega\\0\end{matrix}\right]
=\left[\begin{matrix}0\\\phi\sb\omega\end{matrix}\right],
\]
for all $\omega$ which correspond to solitary waves.
Thus, if we will restrict our attention to functions which are even in $x$,
the dimension of the generalized null space of $\eub{J}\eub{L}$ is
at least two.
Hence, the bifurcation follows the jump in the
dimension of the generalized null space of $\eub{J}\eub{L}$.
Such a jump happens
at a particular value of $\omega$
if one can solve the equation
$\eub{J}\eub{L}\upalpha=\left[\begin{matrix}\p\sb\omega\phi\sb\omega\\0\end{matrix}\right]$.
This leads to the condition
that $\left[\begin{matrix}\p\sb\omega\phi\sb\omega\\0\end{matrix}\right]$
is orthogonal to the null space of the adjoint to $\eub{J}\eub{L}$,
which contains the vector
$\left[\begin{matrix}\phi\sb\omega\\0\end{matrix}\right]$;
this results in
$\langle\phi\sb\omega,\p\sb\omega\phi\sb\omega\rangle
=\p\sb\omega\norm{\phi\sb\omega}\sb{L\sp 2}^2/2=0$.
A slightly more careful analysis
\cite{MR1995870}
based on construction of the moving frame
in the generalized eigenspace of $\lambda=0$
shows that
there are
two real eigenvalues $\pm\lambda\in\R$
that have emerged from $\lambda=0$ when
$\omega$ is such that
$\p\sb\omega\norm{\phi\sb\omega}\sb{L\sp 2}^2$ becomes positive,
leading to a linear instability of the corresponding solitary wave.
The opposite condition,
\begin{equation}\label{dvk-vk}
\p\sb\omega\norm{\phi\sb\omega}\sb{L\sp 2}^2<0,
\end{equation}
is the Vakhitov-Kolokolov stability criterion
which guarantees the absence of nonzero real eigenvalues
for the nonlinear Schr\"odinger equation.
It appeared in \cite{VaKo,MR723756,MR901236}
in relation to linear and orbital stability
of solitary waves.
The above approach fails
for the nonlinear Dirac equation
since $\eur{L}\sb{-}$ is no longer positive-definite.

\bigskip

For the completeness,
let us present a more precise form
of the Vakhitov-Kolokolov stability criterion \cite{VaKo}.

\begin{lemma}[Vakhitov-Kolokolov stability criterion]
\label{dvk-lemma-vk}
There is $\lambda\in\sigma\sb{p}(\eub{J}\eub{L})$,
$\lambda>0$,
where $\eub{J}\eub{L}$ is the linearization
\eqref{dvk-nls-lin}
at the solitary wave
$\phi\sb\omega(x)e^{-i\omega t}$,
if and only if
$\frac{d}{d\omega}\norm{\phi\sb\omega}\sb{L\sp 2}^2>0$
at this value of $\omega$.
\end{lemma}

\begin{proof}
We follow \cite{VaKo}.
Assume that there is $\lambda\in\sigma\sb d(\eub{J}\eub{L})$, $\lambda>0$.
The relation
$(\eub{J}\eub{L}-\lambda)\Xi=0$
implies that
$\lambda^2\Xi\sb 1=-\eur{L}\sb{-} \eur{L}\sb{+}\Xi\sb 1$.
It follows that $\Xi\sb 1$ is orthogonal to the kernel of
the selfadjoint operator $\eur{L}\sb{-}$ (which is spanned by $\phi\sb\omega$):
\[
\langle\phi,\Xi\sb 1\rangle
=-\frac{1}{\lambda^2}\langle\phi,-\eur{L}\sb{-} \eur{L}\sb{+} \Xi\sb 1\rangle
=-\frac{1}{\lambda^2}\langle \eur{L}\sb{-}\phi,-\eur{L}\sb{+} \Xi\sb 1\rangle=0,
\]
hence
there is $\eta\in L\sp 2(\R,\C)$
such that $\Xi\sb 1=\eur{L}\sb{-} \eta$
and $\lambda^2 \eta=-\eur{L}\sb{+} \Xi\sb 1$.
Thus, the inverse to $\eur{L}\sb{-}$ can be applied:
$\lambda^2 \eur{L}\sb{-}^{-1}\Xi\sb 1=-\eur{L}\sb{+} \Xi\sb 1$.
Then
\[
\lambda^2\langle \eta,\eur{L}\sb{-} \eta\rangle
=
-\langle \Xi\sb 1,\eur{L}\sb{+} \Xi\sb 1\rangle.
\]
Since $\eur{L}\sb{-}$ is positive-definite and $\eta\notin\ker \eur{L}\sb{-}$,
it follows that
$\langle \eta,\eur{L}\sb{-} \eta\rangle>0$.
Since $\lambda>0$, $\langle \Xi\sb 1,\eur{L}\sb{+} \Xi\sb 1\rangle<0$,
therefore the quadratic form $\langle\cdot,\eur{L}\sb{+}\cdot\rangle$
is not positive-definite on vectors orthogonal to $\phi\sb\omega$.
According to Lagrange's principle,
the function $r$ corresponding to the minimum of
$\langle r,\eur{L}\sb{+} r\rangle$
under conditions $\langle r,\phi\sb\omega\rangle=0$
and
$\langle r,r\rangle=1$
satisfies
\begin{equation}\label{dvk-r-sat}
\eur{L}\sb{+} r=\alpha r+\beta\phi\sb\omega,
\qquad
\alpha,\ \beta\in\R.
\end{equation}
Since $\langle r,\eur{L}\sb{+} r\rangle=\alpha$,
we need to know whether $\alpha$ could be negative.
Since $\eur{L}\sb{+}\p\sb x\phi\sb\omega=0$,
one has
$\lambda\sb 1=0\in\sigma\sb p(\eur{L}\sb{+})$.
Due to $\p\sb x\phi\sb\omega$
vanishing at one point ($x=0$),
there is exactly one negative eigenvalue of $\eur{L}\sb{+}$,
which we denote by $\lambda\sb 0\in\sigma\sb p(\eur{L}\sb{+})$.
(This eigenvalue corresponds to some non-vanishing eigenfunction.)
Note that
$\beta\ne 0$, or else $\alpha$ would
have to be equal to $\lambda\sb 0$,
with $r$ the corresponding eigenfunction of $\eur{L}\sb{+}$,
but then $r$, having to be nonzero,
could not be orthogonal to $\phi\sb\omega$.
Denote $\lambda\sb 2=\inf(\sigma(\eur{L}\sb{+})\cap\R\sb{+})>0$.
Let us consider
$f(z)=\langle\phi\sb\omega,(\eur{L}\sb{+}-z)^{-1}\phi\sb\omega\rangle$,
which is defined and is smooth
for $z\in(\lambda\sb 0,\lambda\sb 2)$.
(Note that $f(z)$ is defined for $z=\lambda\sb 1=0$
since the corresponding eigenfunction $\p\sb x\phi\sb\omega$
is odd while $\phi\sb\omega$ is even.)
If $\alpha<0$,
then, by \eqref{dvk-r-sat},
we would have
$f(\alpha)=\langle\phi\sb\omega,(\eur{L}\sb{+}-\alpha)^{-1}\phi\sb\omega\rangle
=\frac{1}{\beta}\langle\phi\sb\omega,r\rangle=0$,
and since $f'(z)>0$,
one has $f(0)>0$.
On the other hand,
$f(0)=\langle\phi\sb\omega,\eur{L}\sb{+}^{-1}\phi\sb\omega\rangle
=\langle\phi\sb\omega,\p\sb\omega\phi\sb\omega\rangle
=\frac 1 2\frac{d}{d\omega}\int\sb{\R}\abs{\phi\sb\omega(x)}^2\,dx$.
Therefore, the linear instability
leads to $\alpha<0$,
which results in
$\frac{d}{d\omega}\int\sb{\R}\abs{\phi\sb\omega(x)}^2\,dx>0$.

Alternatively,
let $\frac{d}{d\omega}\norm{\phi\sb\omega}\sb{L\sp 2}^2>0$.
We consider the function
$f(z)=\langle\phi\sb\omega,(\eur{L}\sb{+}-z)^{-1}\phi\sb\omega\rangle$,
$z\in\rho(\eur{L}\sb{+})$.
Since $f(0)=\langle\phi\sb\omega,\eur{L}\sb{+}^{-1}\phi\sb\omega\rangle>0$,
$f'(z)>0$,
and
$\lim\limits\sb{z\to\lambda\sb 0+}f(z)=-\infty$
(where $\lambda\sb 0<0$ is the smallest eigenvalue of $\eur{L}\sb{+}$),
there is $\alpha\in(\lambda\sb 0,0)\subset\rho(\eur{L}\sb{+})$ such that
$f(\alpha)=\langle\phi\sb\omega,(\eur{L}\sb{+}-\alpha)^{-1}\phi\sb\omega\rangle=0$.
Then we define $r=(\eur{L}\sb{+}-\alpha)^{-1}\phi\sb\omega$.
Since
$\langle \phi\sb\omega,r\rangle=f(\alpha)=0$,
there is $\eta$ such that $r=\eur{L}\sb{-} \eta$.
It follows that the quadratic form
$\eur{L}\sb{-}^{1/2}\eur{L}\sb{+} \eur{L}\sb{-}^{1/2}$
is not positive definite:
\[
\langle \eur{L}\sb{-}^{\frac 1 2}\eta,
(\eur{L}\sb{-}^{\frac 12}\eur{L}\sb{+} \eur{L}\sb{-}^{\frac 12})\eur{L}\sb{-}^{\frac 12}\eta\rangle
=
\langle r,\eur{L}\sb{+} r\rangle
=
\langle r,(\alpha r+\phi\sb\omega)\rangle
=
\alpha\langle r,r\rangle<0.
\]
Thus, there is $\lambda>0$ such that
$-\lambda^2\in\sigma(\eur{L}\sb{-}^{1/2}\eur{L}\sb{+} \eur{L}\sb{-}^{1/2})$;
then also
$-\lambda^2\in\sigma(\eur{L}\sb{-} \eur{L}\sb{+})$.
Let $\xi$ be the corresponding eigenvector,
$\eur{L}\sb{-} \eur{L}\sb{+} \xi=-\lambda^2 \xi$;
then
$
\left[\begin{matrix}
0&\eur{L}\sb{-}\\-\eur{L}\sb{+}&0
\end{matrix}\right]
\left[\begin{matrix}
\xi\\-\frac 1 \lambda \eur{L}\sb{+} \xi
\end{matrix}\right]
=
\lambda
\left[\begin{matrix}
\xi\\-\frac 1 \lambda \eur{L}\sb{+} \xi
\end{matrix}\right],
$
hence $\lambda\in\sigma(\eub{J}\eub{L})$.
\end{proof}

\bigskip

Our conclusions:

\begin{enumerate}
\item
Point eigenvalues
of the linearized Dirac equation
may bifurcate
(as $\omega$ changes)
from the origin,
when
the dimension of the generalized null space
jumps up
(when the Vakhitov-Kolokolov criterion breaks down).
\item
Since the spectrum of
the linearization
does not have to be a subset of $\R\cup i\R$,
there may also be point eigenvalues
which bifurcate
from the
imaginary axis into the complex plane.
(We do not know particular examples of such behavior
for the nonlinear Dirac equation.)
\item
Moreover,
there may be point eigenvalues already present in the spectra
of linearizations at arbitrarily small solitary waves.
Formally, we could say that these eigenvalues
bifurcate from the essential spectrum of the
free Dirac operator (divided by $i$),
which can be considered as the linearization
of the nonlinear Dirac equation
at the zero solitary wave.
\end{enumerate}

In the present paper
we investigate the first scenario.
The main result
(Lemma~\ref{dvk-lemma-bifurcations-zero})
states that if the Vakhitov-Kolokolov breaks down at some
point $\omega\sb\ast$,
then, generically, the solitary waves with
$\omega$ from an open one-sided neighborhood of $\omega\sb\ast$
are linearly unstable.

We also demonstrate the presence of the eigenvalue $\pm 2\omega i$
in the spectrum of the linearized operator
(Corollary~\ref{dvk-cor-2omega})
and obtain Virial identities,
or Pohozhaev theorem, for the nonlinear Dirac equation
(Lemma~\ref{dvk-lemma-vir}),
which we need for the analysis of the zero eigenvalue
of the linearized operator.


\section{Linearization of the nonlinear Dirac equation}

The nonlinear Dirac equation in $\R^n$
has the form
\begin{equation}\label{dvk-nld-nd}
i\p\sb t\psi
=-i\sum\sb{j=1}\sp{n}\alpha\sb j\p\sb j\psi
+g(\psi\sp\ast\beta\psi)\beta\psi,
\qquad
\psi(x,t)\in\C^N,
\quad
x\in\R^n,
\end{equation}
where
$\p\sb j=\frac{\p}{\p x\sp j}$,
$N$ is even
and $g$ smooth,
with $m:=g(0)>0$.
The Dirac matrices $\alpha\sb j$ and $\beta$
satisfy the relations
\[
\alpha\sb j\alpha\sb k
+\alpha\sb k\alpha\sb j
=2\delta\sb{jk}I\sb N,
\qquad
\alpha\sb j\beta
+\beta\alpha\sb j=0,
\qquad
\beta^2=I\sb N,
\qquad
1\le j,k\le n,
\]
where $I\sb N$ is an $N\times N$ unit matrix.
We will always assume that
$\beta
=\left[\begin{matrix}I\sb{N/2}&0\\0&-I\sb{N/2}\end{matrix}\right]$.
In the case $n=1$, we assume
$\alpha\sb 1=-\sigma\sb 2$;
in the case $n=3$, one could take
$\alpha\sb j
=\left[\begin{matrix}0&\sigma\sb j\\\sigma\sb j&0\end{matrix}\right]$,
where $\sigma\sb j$ are the standard Pauli matrices.
Equation \eqref{dvk-nld-nd},
usually with $g(s)=1-s$,
is called the Soler model \cite{PhysRevD.1.2766},
which has been
receiving a lot of attention in theoretical physics
in relation to classical models of elementary particles.

\begin{remark}
In terms of the Dirac $\gamma$-matrices,
equation (\ref{dvk-nld-nd}) takes the
explicitly relativistically-invariant form
$
i\gamma\sp \mu\p\sb \mu \psi=g(\psi\sp\ast\beta\psi)\psi,
$
where
$
\ \gamma\sp 0=\beta$,
$\ \gamma\sp j=\beta\alpha\sb{j}$,
$\ \p\sb 0=\p\sb t$,
$\ \p\sb j=\p\sb{x\sb{j}}$.
\end{remark}

\begin{definition}
The solitary waves are solutions
to \eqref{dvk-nld-nd}
of the form
$\phi\sb\omega(x)e^{-i\omega t}$,
$\phi\sb\omega\in H\sp 1(\R^n,\C^N)$,
$\omega\in\R$.
\end{definition}

Below, we assume that there are solitary waves
for $\omega$ from some nonempty set $\Omega\subset\R$:
\begin{equation}\label{dvk-sw}
\phi\sb\omega e^{-i\omega t},
\qquad
\omega\in\Omega\subset\R,
\qquad
\phi\sb\omega\in H\sp 1(\R^n,\C^N),
\end{equation}
with $\phi\sb\omega$ smoothly depending on $\omega$.

\medskip

We will not indicate the dependence on $\omega$ explicitly,
and will write $\phi$ instead of $\phi\sb\omega$.

\medskip

The profile $\phi$ of a stationary wave satisfies
the stationary nonlinear Dirac equation
\begin{equation}\label{dvk-nlds}
\mathcal{L}\sb{-}\phi
:=(-i\sum\sb{j=1}\sp{n}\alpha\sb j\p\sb j-\omega+g(\phi\sp\ast\beta\phi)\beta)\phi=0.
\end{equation}
The energy and charge functionals
corresponding to the nonlinear Dirac equation \eqref{dvk-nld-nd}
are given by
\[
E(\psi)
=\int\sb{R^n}\left(
-i\sum\sb{j=1}\sp{n}\psi\sp\ast\alpha\sb{j}\p\sb{j}\psi
+G(\psi\sp\ast\beta\psi)\right)\,d^n x,
\qquad 
Q(\psi)
=
\int\sb{\R^n}\psi\sp\ast\psi\,d^n x,
\]
where $G(s)$ is the antiderivative of $g(s)$
which satisfies $G(0)=0$.
$Q(\psi)$ is the charge functional
which is (formally) conserved
for solutions to \eqref{dvk-nld-nd}
due to the $\mathbf{U}(1)$-invariance.
The nonlinear Dirac equation
\eqref{dvk-nld-nd}
can be written in the Hamiltonian form
as
$\p\sb t\Im\psi=-\frac 1 2\delta\sb{\Re\psi}E$,
$\p\sb t\Re\psi=\frac 1 2\delta\sb{\Im\psi}E$,
or simply
$ \dot\psi=-i\delta\sb{\psi\sp\ast}E$.
The relation
\eqref{dvk-nlds}
satisfied by the profile of the solitary wave
$\phi(x)e^{-i\omega t}$
can be written as
\begin{equation}\label{dvk-constrained-minimum}
E'(\phi)=\omega Q'(\phi),
\end{equation}
where the primes denote the Fr\'ech\'et derivative
of the functionals $E(\psi)$, $Q(\psi)$
with respect to $(\Re\psi,\Im\psi)$.

Let us write the solution in the form
$
\psi(x,t)=(\phi(x)+\rho(x,t))e^{-i\omega t},
$
$\rho(x,t)\in\C^{N}$.
The linearized equation on $\rho$ is given by
\begin{equation}\label{dvk-nld-nd-l}
\dot{\rho}
=\mathcal{J}\mathcal{L}\rho,
\end{equation}
where $\mathcal{J}$ corresponds to a multiplication by $1/i$
and
\[
\mathcal{L}\rho
=
\mathcal{L}\sb{-}\rho
+
2g'(\phi\sp\ast\beta\phi)
\beta\phi
\Re(\phi\sp\ast\beta\rho).
\]
Note that,
because of the presence of $\Re(\phi\sp\ast\beta\rho)$,
the action of $\mathcal{L}$
on $\rho\in\C^N$
is $\R$-linear but not $\C$-linear.
Because of this,
it is convenient
to write it as an operator
$\eub{L}$
acting on vectors from $\R^{2N}$;
then \eqref{dvk-nld-nd-l} takes the following form:
\begin{equation}\label{dvk-drs}
\p\sb t
\eub{R}
=
\eub{J}\eub{L}
\eub{R}
,
\qquad
\eub{L}
=
\eub{J}\eub{A}\sb j\p\sb j-\omega+g\eub{B}
+
2g'
\eub{B}
\varPhi
\langle\eub{B}\varPhi,\,\cdot\,\rangle\sb{\R^{2N}},
\qquad
\eub{R}(x,t)
=\left[\begin{matrix}\Re\rho\\\Im\rho\end{matrix}\right],
\end{equation}
where
$g=g(\phi\sp\ast\beta\phi)$,
$g'=g'(\phi\sp\ast\beta\phi)$
and
\begin{equation}\label{dvk-def-aleph}
\varPhi=\left[\begin{matrix}\Re\phi\\\Im\phi\end{matrix}\right],
\qquad
\eub{J}=\left[\begin{matrix}0&I\sb N\\-I\sb N&0\end{matrix}\right],
\qquad
\eub{A}\sb j
=\left[\begin{matrix}
\Re\alpha\sb j&\!\!-\Im\alpha\sb j\\\Im\alpha\sb j&\Re\alpha\sb j
\end{matrix}\right],
\qquad
\eub{B}=\left[\begin{matrix}\beta&0\\0&\beta\end{matrix}\right].
\end{equation}
Note that
$\eub{J}$, $\eub{A}\sb j$, and $\eub{B}$
correspond to multiplication by $-i$, $\alpha\sb j$,
and $\beta$
under the $\C^N\leftrightarrow\R^{2N}$ correspondence. 

\begin{remark}\label{dvk-remark-l0l1}
When $n=1$,
one can take
$\alpha\sb 1=-\sigma\sb 2$
(so that $N=2$).
Then
$\eub{L}$ has a particularly simple form
since
$\phi\in\C^2$ can be chosen valued in $\R^2$:
$
\eub{L}
=
\left[\begin{matrix}
\eur{L}\sb{+}&0
\\
0&\eur{L}\sb{-}
\end{matrix}\right]
$,
with
$\eur{L}\sb{-}=i\sigma\sb 2\p\sb x-\omega+g\beta$,
$
\eur{L}\sb{+}=\eur{L}\sb{-}
+2g'
\beta\phi
\,\langle\beta\phi,\,\cdot\,\rangle\sb{\C^2}
$.
The numerical and analytical study of spectra
of $\eur{L}\sb{-}$, $\eur{L}\sb{+}$
in this case
is contained in
\cite{dirac-1d-arxiv}.
\end{remark}

\begin{lemma}\label{dvk-lemma-sigma-c}
$
\spec\sb{ess}(\eub{L})
=\spec\sb{ess}(\mathcal{L})
=\R\backslash(-m-\omega,m-\omega);
$
$\sigma\sb{ess}(\eub{J}\eub{L})=i(\R\backslash(-m+\omega,m-\omega))$.
\end{lemma}

\begin{proof}
One has
$\sigma(-i\alpha\sb j\p\sb j+\beta m)=\R\backslash(-m,m)$.
Note also that
$
(-i\alpha\sb j\p\sb j+\beta m)^2=-\Delta+m^2$
has a spectrum $[m^2,\infty)$.
Taking into account that the symbol of
$\mathcal{L}$ at $\abs{x}\to\infty$
is
$
\alpha\sb j\xi\sb j
+\beta m-\omega,
$
one concludes that
$
\spec\sb{ess}(\eub{L})
=\spec\sb{ess}(\mathcal{L})
=\R\backslash(-m-\omega,m-\omega).
$
Since the eigenvalues of $\eub{J}$ are $\pm i$,
corresponding to clock- and counterclockwise rotations in $\C$,
one deduces that
$\sigma\sb{ess}(\eub{J}\eub{L})=i(\R\backslash(-m+\omega,m-\omega))$.
\end{proof}

\begin{lemma}
The null space of $\eub{J}\eub{L}$ is given by
$
\ker\eub{J}\eub{L}
=\mathop{\rm Span}\left\{
\eub{J}\varPhi,
\ \p\sb k\varPhi
\sothat
1\le k\le n
\right\}.
$
\end{lemma}

\begin{proof}
Recall that
\[
\mathcal{L}\sb{-}
=-i\sum\sb{j=1}\sp{n}\alpha\sb j\p\sb j-\omega+g(\phi\sp\ast\beta\phi)\beta,
\qquad
\mathcal{L}
=-i\sum\sb{j=1}\sp{n}\alpha\sb j\p\sb j-\omega+g(\phi\sp\ast\beta\phi)\beta
+
2g'(\phi\sp\ast\beta\phi)
\beta\phi
\Re(\phi\sp\ast\beta\,\,\cdot\,).
\]
Since $\phi(x)\in\C^N$ satisfies the stationary
nonlinear Dirac equation \eqref{dvk-nlds},
we get:
\begin{equation}\label{dvk-ast1}
\mathcal{L}(-i\phi)
=
\mathcal{L}\sb{-}(-i\phi)
+2\Re(\phi\sp\ast\beta(-i\phi))=0.
\end{equation}
Taking the derivative of \eqref{dvk-nlds}
with respect to $x\sb k$ yields
\begin{equation}\label{dvk-ldxf}
-i\sum\sb{j=1}\sp{n}\alpha\sb j\p\sb j(\p\sb k\phi)
+g(\phi\sp\ast\beta\phi)\beta\p\sb k\phi
+g'(\phi\sp\ast\beta\phi)
(\p\sb k\phi\sp\ast\beta\phi
+\phi\sp\ast\beta\p\sb k\phi)\beta\phi
-\omega\p\sb k\phi
=\mathcal{L}\p\sb k\phi=0.
\end{equation}
\end{proof}

\begin{lemma}
Let $\alpha\sb 0$
be an hermitian matrix
anticommuting with
$\alpha\sb j$, $1\le j\le n$,
and with $\beta$.
Then
$\alpha\sb 0\phi$
is an eigenfunction of
$\mathcal{L}\sb{-}$
and of $\mathcal{L}$,
corresponding to the eigenvalue $\lambda=-2\omega$.
\end{lemma}

\begin{remark}
If $n=3$, one can take $\alpha\sb 0=\alpha\sb 1\alpha\sb 2\alpha\sb 3\beta$.
\end{remark}

\begin{proof}
Since $\alpha\sb 0$
anticommutes with $\alpha\sb j$ ($1\le j\le n$)
and with $\beta$,
and taking into account \eqref{dvk-nlds},
we have:
\[
\mathcal{L}\sb{-}\alpha\sb 0\phi
=
(-i\sum\sb{j=1}\sp{n}\alpha\sb j\p\sb j-\omega+g(\phi\sp\ast\beta\phi)\beta)\alpha\sb 0\phi
=
\alpha\sb 0(i\sum\sb{j=1}\sp{n}\alpha\sb j\p\sb j-\omega-g(\phi\sp\ast\beta\phi)\beta)\phi
=\alpha\sb 0(-\mathcal{L}\sb{-}-2\omega)\phi
=-2\omega\alpha\sb 0\phi.
\]
Since
$\alpha\sb 0$ and $\beta$ are Hermitian,
$
2\Re[\phi\sp\ast\beta\alpha\sb 0\phi]
=
\phi\sp\ast\beta\alpha\sb 0\phi
+
\overline{\phi\sp\ast\beta\alpha\sb 0\phi}
=
\phi\sp\ast\{\beta,\alpha\sb 0\}\phi
=0;
$
therefore,
one also has
$
\mathcal{L}\alpha\sb 0\phi
=
\mathcal{L}\sb{-}\alpha\sb 0\phi
=-2\omega\alpha\sb 0\phi.
$
\end{proof}

It follows that the linearization operator
has an eigenvalue $2\omega i$:
\[
2\omega i\in\sigma\sb p(\mathcal{J}\mathcal{L})=\sigma\sb p(\eub{J}\eub{L}).
\]
Since $\sigma(\eub{J}\eub{L})$ is symmetric with respect to
$\R$ and $i\R$,
for any $g(s)$
in \eqref{dvk-nld-nd}
and in any dimension $n\ge 1$,
we have:

\begin{corollary}\label{dvk-cor-2omega}
$\pm 2\omega i$
are $L\sp 2$ eigenvalues of $\eub{J}\eub{L}$.
\end{corollary}

\begin{remark}
For $\abs{\omega}>m/3$,
the eigenvalues $\pm 2\omega i$
are embedded in the essential spectrum.
This is in contradiction
with the Hypothesis (H:6)
in \cite{2011arXiv1103.4452B}
on the absence of eigenvalues
embedded in the essential spectrum,
although we hope that this difficulty
could be dealt with using a minor change in the proof.
\end{remark}

\begin{remark}
The result of Corollary~\ref{dvk-cor-2omega} takes place for any nonlinearity
$g(\psi\sp\ast\beta\psi)$
and in any dimension.
The spatial dimension $n$ and the number
of components of $\psi$
could be such that there is no matrix
$\alpha\sb 0$ which anticommutes with
$\alpha\sb j$, $1\le j\le n$, and with $\beta$;
then the eigenvector corresponding to $\pm 2\omega i$
can be constructed using the spatial reflections.
\end{remark}




\section{Virial identities}

When studying the bifurcation
of eigenvalues from $\lambda=0$,
we will need some conclusions about
the generalized null space
of the linearized operator.
We will draw these conclusions from the Virial identities,
which are also known
as the Pohozhaev theorem \cite{MR0192184}.
In the context of the nonlinear Dirac equations,
similar results were presented in \cite{MR1344729}.

\begin{lemma}\label{dvk-lemma-var}
For a differentiable family
$\lambda\mapsto\phi\sb\lambda\in H\sp 1(\R^n)$,
$\phi\sb\lambda\at{\lambda=1}=\phi$,
one has
$
\p\sb\lambda E(\phi\sb\lambda)
\at{\lambda=1}
=\omega\p\sb\lambda Q(\phi\sb\lambda)\at{\lambda=1}.
$
\end{lemma}

\begin{proof}
This immediately follows from \eqref{dvk-constrained-minimum}.
\end{proof}

We split the Hamiltonian $E(\psi)$ into 
$
E(\psi)=T(\psi)+V(\psi)
=\sum\limits\sb{j=1}\sp{n}T\sb j(\psi)+V(\psi),
$
where
\[
T\sb j(\psi)
=-i\int\sb{\R^n}\psi\sp\ast\alpha\sb{j}\p\sb{j}\psi\,d^n x
\quad
\mbox{(no summation in $j$)},
\qquad
T(\psi)=\sum\sb{j=1}\sp{n}T\sb j(\psi),
\qquad
V(\psi)=\int\sb{\R^n} G(\psi\sp\ast\beta\psi)\,d^n x.
\]

\begin{lemma}[Pohozhaev Theorem for the nonlinear Dirac equation]
\label{dvk-lemma-vir}
For each solitary wave
$\phi(x)e^{-i\omega t}$,
there are the following relations:
\[
\frac{n-1}{n}T(\phi)+V(\phi)=\omega Q(\phi),
\qquad
T\sb j(\phi)
=\frac{1}{n}T(\phi)
=
\int\sb{\R^n}
\big(
G(\phi\sp\ast\beta\phi)
-
\phi\sp\ast\beta\phi
\,
g(\phi\sp\ast\beta\phi)
\big)\,dx,
\quad
1\le j\le n.
\]
\end{lemma}

\begin{proof}
We set
$\phi\sb\lambda(x)=\phi(x\sb 1/\lambda,x\sb 2,\ldots,x\sb n)$.
Since 
$T\sb 1(\phi\sb\lambda)=T\sb 1(\phi)$,
$T\sb j(\phi\sb\lambda)=\lambda T\sb j(\phi)$ for $2\le j\le n$,
$V(\phi\sb\lambda)=\lambda V(\phi)$,
$Q(\phi\sb\lambda)=\lambda Q(\phi)$,
we can use Lemma~\ref{dvk-lemma-var}
to obtain the following relation
(``Virial theorem''):
\begin{equation}\label{dvk-vt1}
\sum\sb{j=2}\sp{n}T\sb j(\phi)+V(\phi)
=
\p\sb\lambda\at{\lambda=1}
\Big(\sum\sb{j=1}\sp{n}T\sb j(\phi\sb\lambda)+V(\phi\sb\lambda)\Big)
=
\omega\p\sb\lambda\at{\lambda=1}Q(\phi\sb\lambda)
=\omega Q(\phi).
\end{equation}
Similarly, rescaling in $x\sb j$, $1\le j\le n$,
we conclude that $T\sb 1(\phi)=\ldots=T\sb n(\phi)=\frac{1}{n}T(\phi)$,
and \eqref{dvk-vt1} gives
\begin{equation}\label{dvk-vt2}
\frac{n-1}{n}T(\phi)+V(\phi)=\omega Q(\phi).
\end{equation}
Moreover, the relation \eqref{dvk-nlds}
yields
$
T(\phi)+
\int\sb{\R^n}
\phi\sp\ast\beta\phi\,g(\phi\sp\ast\beta\phi)
\,dx
=\omega Q(\phi).
$
Together with \eqref{dvk-vt2},
this gives the desired relation
$
\frac{1}{n}T(\phi)
=
\int\sb{\R^n}
\big(
G(\phi\sp\ast\beta\phi)
-
\phi\sp\ast\beta\phi\,
g(\phi\sp\ast\beta\phi)
\big)\,dx.
$
\end{proof}

\begin{remark}\label{dvk-remark-positive}
For all nonlinearities
for which the existence
of solitary wave solutions
is proved in \cite{MR1344729},
one has $G(s)-sG'(s)>0$ for all $s\in\R$ except finitely many points
(e.g. $s=0$);
hence for these solitary waves one has $T(\phi)>0$.
In particular, for
$G(s)=s-\frac{s^2}{2}$,
one has
$T(\phi)=n\int\sb{\R^n}\frac{\abs{\phi\sp\ast\beta\phi}^2}{2}\,dx>0$.
\end{remark}

\section{Bifurcations from $\lambda=0$}
\label{dvk-sect-bifurcations-zero}

\begin{lemma}\label{dvk-lemma-bifurcations-zero}
Assume that
the nonlinearity satisfies
the following inequality
(see Remark~\ref{dvk-remark-positive}):
\begin{equation}\label{dvk-gsg}
\quad
\frac{n+1}{n}
G(s)-s G'(s)>0,
\quad
s\in\R,
\quad
s\ne 0.
\end{equation}
Further, assume that $\phi\sp\ast\phi$ and $\phi\sp\ast\beta\phi$
are spherically symmetric
and that
\[
\mathscr{N}(\eub{L})
=
\mathop{\rm Span}\left\{
\ \eub{J}\varPhi,
\ \p\sb k\varPhi
\sothat
1\le k\le n
\right\};
\qquad
\dim \mathscr{N}(\eub{L})=n+1.
\]
If $\p\sb\omega Q(\phi)\ne 0$,
then the generalized null space
of $\eub{J}\eub{L}$ is given by
\[
\mathscr{N}\sb g(\eub{J}\eub{L})
=
\mathop{\rm Span}\left\{
\ \eub{J}\varPhi,\ \p\sb\omega\varPhi,
\ \p\sb k\varPhi,\ \eub{A}\sb k\varPhi
\sothat
1\le k\le n
\right\};
\qquad
\dim \mathscr{N}\sb g(\eub{J}\eub{L})=2n+2.
\]
If
$\p\sb\omega Q(\phi)$
vanishes at $\omega\sb\ast$,
then
$\dim \mathscr{N}\sb g(\eub{J}\eub{L}\at{\omega\sb\ast})\ge 2n+4$.
Moreover,
generically,
there is an eigenvalue
$\lambda\in\sigma\sb{d}(\eub{J}\eub{L})$
with $\Re\lambda>0$
for $\omega$
from an open one-sided neighborhood
of $\omega\sb\ast$.
\end{lemma}

\begin{remark}
The assumption that
$\phi\sp\ast\phi$ and $\phi\sp\ast\beta\phi$
are spherically symmetric
is satisfied by the ansatz
\[
\phi(x)=\left[\begin{matrix}
g(r)\left[\begin{matrix}1\\0\end{matrix}\right]
\\
if(r)\left[\begin{matrix}\cos\theta\\e^{i\phi}\sin\theta\end{matrix}\right]
\end{matrix}\right]
\]
used in e.g. \cite{wakano-1966,PhysRevD.1.2766,MR1344729,MR1386737}.
\end{remark}

\begin{proof}
Taking the derivative of \eqref{dvk-nlds}
with respect to $\omega$,
we get
\begin{equation}\label{dvk-ldof}
-i\sum\sb{j=1}\sp{n}\alpha\sb j\p\sb j\p\sb\omega\phi
+g(\phi\sp\ast\beta\phi)\beta\p\sb\omega\phi
+g'(\phi\sp\ast\beta\phi)
(\p\sb\omega\phi\sp\ast\beta\phi
+\phi\sp\ast\beta\p\sb\omega\phi)\beta\phi
-\omega\p\sb\omega\phi
-\phi
=\mathcal{L}\p\sb\omega\phi-\phi=0.
\end{equation}
Since
$
\phi\sp\ast\beta\alpha\sb k\phi
+(\alpha\sb k\phi)\sp\ast\beta\phi
=\phi\sp\ast\{\beta,\alpha\}\phi=0$,
we have
\begin{equation}\label{dvk-hap1}
\mathcal{L}(\alpha\sb k\phi)
=\mathcal{L}\sb{-}(\alpha\sb k\phi)
=
-2i\p\sb k\phi
-2\omega\alpha\sb k\phi
-\alpha\sb k\mathcal{L}\sb{-}\phi
=-2i\p\sb k\phi
-2\omega\alpha\sb k\phi.
\end{equation}
Similarly,
since
$
\phi\sp\ast\beta(i x\sb k\phi)
+(i x\sb k\phi)\sp\ast\beta\phi=0$,
\begin{equation}\label{dvk-hap2}
\mathcal{L}(i x\sb k\phi)
=\mathcal{L}\sb{-}(ix\sb k\phi)
=\alpha\sb k\phi
+i x\sb k\mathcal{L}\sb{-}\phi
=\alpha\sb k\phi.
\end{equation}
Using \eqref{dvk-hap1} and \eqref{dvk-hap2},
we have
\begin{equation}\label{dvk-ast2}
\mathcal{L}(\alpha\sb k\phi+2\omega i x\sb k\phi)=-2i\p\sb k\phi.
\end{equation}
Until the end of this section,
it will be more convenient for us
to work in terms of
$\varPhi\in\R^{2N}$ (see \eqref{dvk-def-aleph}).
We summarize the above relations
\eqref{dvk-ast1},
\eqref{dvk-ldxf}, \eqref{dvk-ldof},
\eqref{dvk-hap1}, \eqref{dvk-hap2},
and
\eqref{dvk-ast2}
as follows:
\[
\eub{L} \eub{J}\varPhi=0,
\qquad
\eub{L}\p\sb k\varPhi=0,
\qquad
\eub{L}\p\sb\omega\varPhi=\varPhi,
\qquad
\eub{L}(-x\sb k \eub{J}\varPhi)
=\eub{A}\sb k\varPhi,
\qquad
\eub{L}(\eub{A}\sb k\varPhi-2\omega x\sb k \eub{J}\varPhi)=2\eub{J}\p\sb k\varPhi,
\]
where
\[
\eub{L}=\eub{J}\eub{A}\sb j\p\sb j-\omega+g\eub{B}
+2g'\eub{B}\varPhi
\langle\varPhi,\eub{B}\,\cdot\,\rangle\sb{\R^{2N}},
\qquad
\eub{L}\sb 0=\eub{J}\eub{A}\sb j\p\sb j-\omega+g\eub{B},
\]
with
$g=g(\phi\sp\ast\beta\phi)=g(\varPhi\sp\ast\eub{B}\varPhi)$,
$g'=g'(\phi\sp\ast\beta\phi)=g'(\varPhi\sp\ast\eub{B}\varPhi)$.
There are no $F\sb k$
such that
$
\eub{J}\eub{L} F\sb k=\eub{A}\sb k\varPhi-2\omega x\sb k \eub{J}\varPhi.
$
Indeed,
checking the orthogonality
of $\eub{A}\sb k\varPhi-2\omega x\sb k \eub{J}\varPhi$
with respect to
$\mathscr{N}((\eub{J}\eub{L})\sp\ast)
=\mathop{\rm Span}\left\{\varPhi,\eub{J}\p\sb k\varPhi\right\}$,
we have:
\begin{equation}\label{dvk-nsj}
\langle\eub{A}\sb k\varPhi-2\omega x\sb k \eub{J}\varPhi,\eub{J}\p\sb k\varPhi\rangle
=\int\sb{\R^n}
\phi\sp\ast(-i\sum\sb{k=1}\sp{n}\alpha\sb k\p\sb k)\phi\,dx
+\omega\int\sb{\R^n}\phi\sp\ast\phi\,dx
=\frac{T(\phi)}{n}+\omega Q(\phi).
\end{equation}
Note that there is no summation in $k$ in \eqref{dvk-nsj}.
By Lemma~\ref{dvk-lemma-vir},
the right-hand side of \eqref{dvk-nsj}
is equal to
\begin{equation}\label{dvk-this}
\frac{T}{n}+\omega Q
=
\frac{T}{n}+\frac{(n-1)T}{n}+V
=T+V
=\int\sb{\R^n}(n(G(\rho)-\rho G'(\rho))+G(\rho))\,dx,
\end{equation}
where
$\rho(x)=\phi\sp\ast(x)\beta\phi(x)$.
By \eqref{dvk-gsg},
the right-hand side of \eqref{dvk-this}
is strictly positive.

\bigskip

Due to Lemma~\ref{dvk-lemma-sigma-c},
we can choose
a small counterclockwise-oriented circle
$\gamma$
centered at $\lambda=0$
such that at $\omega=\omega\sb\ast$
the only part of the spectrum
$\sigma(\eub{J} \eub{L})$
inside $\gamma$
is the eigenvalue $\lambda=0$.
Assume that $\mathcal{O}$ is an open neighborhood
of $\omega\sb\ast$
small enough so that
$\sigma(\eub{J}\eub{L})$
does not intersect $\gamma$
for $\omega\in\mathcal{O}$.
Define
\begin{equation}\label{dvk-def-pd}
P\sb 0
=\frac{1}{2\pi i}\oint\sb\gamma\frac{d\lambda}{\lambda-\eub{J}\eub{L}},
\qquad
\omega\in\mathcal{O}.
\end{equation}
For each $\omega\in\mathcal{O}$,
$P\sb 0$
is a projection onto a
finite-dimensional vector space $\mathscr{X}:=\Range(P\sb 0)\subset L\sp 2(\R^n)$.
The operator $P\sb 0\eub{J}\eub{L}$ is bounded
(since $P\sb 0$ is smoothing of order one) and with the finite-dimensional range.
Applying the Fredholm alternative to $P\sb 0\eub{J}\eub{L}$,
we conclude that
there is ${\bm E}\sb 3$ such that
$P\sb 0\eub{J}\eub{L} {\bm E}\sb 3=\p\sb\omega\varPhi$
(hence $\eub{J}\eub{L} {\bm E}\sb 3=\p\sb\omega\varPhi$)
if and only if $\p\sb\omega Q(\phi)=0$.
Indeed, one can check that
it is precisely in this case that 
$\p\sb\omega\varPhi$ is orthogonal to
$\mathscr{N}((\eub{J}\eub{L})\sp\ast)
=\mathop{\rm Span}\left\{\varPhi,\eub{J}\p\sb k\varPhi\right\}$:
\[
\langle\p\sb\omega\varPhi,\varPhi\rangle
=\frac 1 2\p\sb\omega\int\phi\sp\ast\phi\,dx
=\frac 1 2\p\sb\omega Q(\phi),
\]
\[
2\langle\p\sb\omega\varPhi,\eub{J}\p\sb k\varPhi\rangle
=\langle\p\sb\omega\varPhi,
\eub{L}
(\eub{A}\sb k\varPhi-2\omega x\sb k\eub{J}\varPhi)
\rangle
=\langle\eub{L}\p\sb\omega\varPhi,
(\eub{A}\sb k\varPhi-2\omega x\sb k\eub{J}\varPhi)
\rangle
\]
\[
=\langle\varPhi,
(\eub{A}\sb k\varPhi-2\omega x\sb k\eub{J}\varPhi)
\rangle
=\langle\varPhi,\eub{A}\sb k\varPhi\rangle
-2\omega\langle\varPhi,x\sb k\eub{J}\varPhi\rangle
=\int\sb{\R^n}\phi\sp\ast\alpha\sb k\phi\,dx
=0.
\]
The right-hand side
vanishes since it
is the $k$th component
of the (zero) momentum
of the standing solitary wave.

\begin{remark}
Using \eqref{dvk-nlds},
one can explicitly compute
\[
2\omega\langle\phi,\alpha\sb k\phi\rangle
=
\langle\phi,\alpha\sb k(-i\alpha\sb j\p\sb j+g\beta)\phi\rangle
+
\langle(-i\alpha\sb j\p\sb j+g\beta)\phi,\alpha\sb k\phi\rangle
=
-2i\langle\phi,\p\sb k\phi\rangle
=
-i\int\sb{\R^n}\p\sb k(\phi\sp\ast\phi)\,dx
=0.
\]
\end{remark}

Once there is ${\bm E}\sb 3$
such that $\eub{J}\eub{L}{\bm E}\sb 3=\p\sb\omega\varPhi$,
there is also ${\bm E}\sb 4$ such that
$
\eub{J}\eub{L} {\bm E}\sb 4={\bm E}\sb 3,
$
since
${\bm E}\sb 3$ is orthogonal to the null space
$\mathscr{N}((\eub{J}\eub{L})\sp\ast)
=\mathop{\rm Span}\left\{\varPhi,\eub{J}\p\sb k\varPhi\right\}$:
\[
\langle {\bm E}\sb 3,\varPhi\rangle
=\langle {\bm E}\sb 3,\eub{L}\p\sb\omega\varPhi\rangle
=\langle\eub{L} {\bm E}\sb 3,\p\sb\omega\varPhi\rangle
=-\langle \eub{J}\p\sb\omega\varPhi,\p\sb\omega\varPhi\rangle
=0,
\]
\[
2\langle {\bm E}\sb 3,\eub{J}\p\sb k\varPhi\rangle
=\langle {\bm E}\sb 3,\eub{L}(\eub{A}\sb k\varPhi-2\omega x\sb k\eub{J}\varPhi)\rangle
=\langle \eub{L} {\bm E}\sb 3,\eub{A}\sb k\varPhi-2\omega x\sb k\eub{J}\varPhi\rangle
=-\langle \eub{J}\p\sb\omega\varPhi,\eub{A}\sb k\varPhi-2\omega x\sb k\eub{J}\varPhi\rangle
=0.
\]
To check that the right-hand side
of the second line
is indeed equal to zero,
one needs to take into account the following:
\begin{equation}\label{dvk-inte1}
\langle \eub{J}\p\sb\omega\varPhi,\omega x\sb k \eub{J}\varPhi\rangle
=
\omega\langle\p\sb\omega\varPhi,x\sb k \varPhi\rangle
=\frac{\omega}{2}\p\sb\omega\int\sb{\R^n} x\sb k\phi\sp\ast\phi\,dx
=0,
\end{equation}
\begin{eqnarray}\label{dvk-inte2}
&&
\langle \eub{J}\p\sb\omega\varPhi,\eub{A}\sb k\varPhi\rangle
=\langle \eub{J}\p\sb\omega\varPhi,\eub{L}(-x\sb k \eub{J}\varPhi)\rangle
=\langle \eub{J}\p\sb\omega\varPhi,\eub{L}\sb 0(-x\sb k \eub{J}\varPhi)\rangle
=-\langle\eub{L}\sb 0\p\sb\omega\varPhi,x\sb k\varPhi\rangle
\nonumber
\\
&&
=-\langle\varPhi-2g'(\phi\sp\ast\beta\phi)
\eub{B}\varPhi\langle\varPhi,\eub{B}\p\sb\omega\varPhi\rangle\sb{\R^{2N}},
x\sb k\varPhi\rangle
=-\langle\varPhi,x\sb k\varPhi\rangle
+\langle 2g'(\phi\sp\ast\beta\phi)
\eub{B}\varPhi\langle\varPhi,\eub{B}\p\sb\omega\varPhi\rangle\sb{\R^{2N}},
x\sb k\varPhi\rangle
\nonumber
\\
&&
=-\int\sb{\R^n}x\sb k\phi\sp\ast\phi\,dx
+\p\sb\omega\int\sb{\R^n}x\sb k K(\phi\sp\ast\beta\phi)\,dx=0.
\end{eqnarray}
Above,
$K(s)$ is the antiderivative of
$s g'(s)$ such that $K(0)=0$.
The integrals
in the right-hand sides of \eqref{dvk-inte1} and \eqref{dvk-inte2}
are equal to zero
due to our assumption on the symmetry properties
of $\phi\sp\ast\phi$ and $\phi\sp\ast\beta\phi$.
Note that $\mathcal{L}\sb{-}$ is $\C$-linear,
hence commutes with a multiplication by $i$,
and
therefore $\eub{J}$ and $\eub{L}\sb 0$ commute;
we used this
when deriving \eqref{dvk-inte2}.

We will assume that
\begin{equation}\label{dvk-e3e3-nonzero}
\langle{\bm E}\sb 3,\eub{L}{\bm E}\sb 3\rangle\ne 0.
\end{equation}
Then there is no ${\bm E}\sb 5$ such that
$\eub{J}\eub{L}{\bm E}\sb 5={\bm E}\sb 4$,
since
${\bm E}\sb 4$ is not orthogonal
to the null space $\mathcal{N}((\eub{J}\eub{L})\sp\ast)\ni\varPhi$:
\begin{equation}\label{dvk-e4-phi-nonzero}
\langle{\bm E}\sb 4,\varPhi\rangle
=
\langle{\bm E}\sb 4,\eub{L}\p\sb\omega\varPhi\rangle
=
\langle{\bm E}\sb 4,\eub{L} \eub{J}\eub{L}{\bm E}\sb 3\rangle
=
-\langle{\bm E}\sb 3,\eub{L}{\bm E}\sb 3\rangle\ne 0.
\end{equation}

\begin{remark}
If \eqref{dvk-e3e3-nonzero}
is not satisfied, then
the dimension of the generalized null space
of $\eub{J}\eub{L}$ at $\omega\sb\ast$ may jump
by more than two;
this means that there are more than
two eigenvalues colliding at $\lambda=0$
as $\omega$ passes through $\omega\sb\ast$.
We expect that generically
this scenario does not take place.
\end{remark}

We will
break the finite-dimensional vector space $\mathscr{X}=\Range(P\sb 0)$
into a direct sum
\[
\mathscr{X}=\mathscr{Y}\oplus \mathscr{Z},
\qquad
\mathscr{Z}=\mathop{\rm Span}\left\{\p\sb j\varPhi,\eub{A}\sb j\varPhi\sothat
1\le j\le n\right\},
\]
so that both $\mathscr{Y}$ and $\mathscr{Z}$
are invariant with respect to the action of $\eub{J}\eub{L}$.
(Let us mention that
$P\sb 0$,
$\mathscr{X}$, $\mathscr{Y}$, $\mathscr{Z}$, and $\eub{J}\eub{L}$ depend on $\omega$.)
Set
\begin{equation}
{\bm e}\sb 4(\omega)
=
P\sb 0{\bm E}\sb 4
\in \mathscr{Y},
\qquad
{\bm e}\sb 3(\omega)
=
\eub{J}\eub{L}{\bm e}\sb 4
\in \mathscr{Y}.
\end{equation}
Since $P\sb 0$ continuously depends on $\omega$,
${\bm e}\sb 3(\omega)$ and ${\bm e}\sb 4(\omega)$
are continuous functions of $\omega$.
$
\{{\bm e}\sb 1(\omega),{\bm e}\sb 2(\omega),
{\bm e}\sb 3(\omega),{\bm e}\sb 4(\omega)\},
$
$
\omega\in\mathcal{O},
$
is a frame in the space
$\mathscr{Y}$.
For some
continuous functions
$\sigma\sb 1(\omega)$, $\sigma\sb 2(\omega)$,
$\sigma\sb 3(\omega)$, and $\sigma\sb 4(\omega)$,
there is the relation
\begin{equation}\label{dvk-jh2-e2-e4}
\eub{J}\eub{L}
{\bm e}\sb 3(\omega)
=
\sigma\sb 1(\omega){\bm e}\sb 1(\omega)
+
\sigma\sb 2(\omega){\bm e}\sb 2(\omega)
+
\sigma\sb 3(\omega){\bm e}\sb 3(\omega)
+
\sigma\sb 4(\omega){\bm e}\sb 4(\omega),
\qquad
\omega\in\mathcal{O}.
\end{equation}
Evaluating \eqref{dvk-jh2-e2-e4}
at $\omega\sb\ast$,
we conclude that
\begin{equation}\label{dvk-sigma0100}
\sigma\sb 1(\omega\sb\ast)=0,
\qquad
\sigma\sb 2(\omega\sb\ast)=1,
\qquad
\sigma\sb 3(\omega\sb\ast)=0,
\qquad
\sigma\sb 4(\omega\sb\ast)=0.
\end{equation}
In the frame
$
\{
{\bm e}\sb 1(\omega),
\ {\bm e}\sb 2(\omega),
\ {\bm e}\sb 3(\omega),
\ {\bm e}\sb 4(\omega)
\},
$
the operator $\eub{J}\eub{L}$
restricted onto $\mathscr{Y}$
is represented by the matrix
\begin{equation}\label{dvk-rbm}
M=
\left[\begin{matrix}
0&1&\sigma\sb 1(\omega)&0
\\
0&0&\sigma\sb 2(\omega)&0
\\
0&0&\sigma\sb 3(\omega)&1
\\
\quad 0\quad &0&\sigma\sb 4(\omega)&\quad 0\quad
\end{matrix}\right],
\qquad
\omega\in\mathcal{O},
\end{equation}
which is a $4\times 4$ Jordan block at $\omega\sb\ast$.
Let us investigate its entries.
Pairing \eqref{dvk-jh2-e2-e4}
with $\varPhi=\eub{J}^{-1}{\bm e}\sb 1$
and taking into account that
$
\left\langle
\eub{J}^{-1}{\bm e}\sb 1,(\eub{J}\eub{L})^2{\bm e}\sb 4(\omega)
\right\rangle
=
-\left\langle
\eub{L} {\bm e}\sb 1,\eub{J}\eub{L}{\bm e}\sb 4(\omega)
\right\rangle=0,
$
$\langle \eub{J}^{-1}{\bm e}\sb 1,{\bm e}\sb 1\rangle=0$,
and
$\langle \eub{J}^{-1}{\bm e}\sb 1,{\bm e}\sb 3\rangle
=\langle \eub{J}^{-1}{\bm e}\sb 1,\eub{J}\eub{L}{\bm e}\sb 4\rangle
=-\langle\eub{L}{\bm e}\sb 1,{\bm e}\sb 4\rangle
=0$,
we get:
\begin{equation}\label{dvk-pairing}
0
=
\sigma\sb 2(\omega)
\left\langle
\eub{J}^{-1}{\bm e}\sb 1,{\bm e}\sb 2
\right\rangle
+
\sigma\sb 4(\omega)
\left\langle
\eub{J}^{-1}{\bm e}\sb 1,{\bm e}\sb 4
\right\rangle.
\end{equation}
Define the function
\begin{equation}\label{dvk-def-mu}
\mu(\omega):=-\langle \varPhi,{\bm e}\sb 4\rangle,
\qquad
\omega\in\mathcal{O}.
\end{equation}
Since
\begin{equation}\label{dvk-mu-nonzero}
\mu(\omega\sb\ast)
=-\langle \varPhi,{\bm e}\sb 4\rangle\at{\omega\sb\ast}
=-\langle \varPhi\at{\omega\sb\ast},{\bm E}\sb 4\rangle
=\langle{\bm E}\sb 3,\eub{L}{\bm E}\sb 3\rangle
\ne 0
\end{equation}
by \eqref{dvk-e4-phi-nonzero},
we may take
the open neighborhood $\mathcal{O}$ of $\omega\sb\ast$
to be sufficiently small so that
$\mu(\omega)$ does not vanish for $\omega\in\mathcal{O}$.
Then
\begin{equation}\label{dvk-sigma4-sigma2}
\sigma\sb 4(\omega)
=\frac{\sigma\sb 2(\omega)\p\sb\omega Q(\phi)}{\mu(\omega)},
\qquad
\omega\in\mathcal{O}.
\end{equation}
Let us show that $\sigma\sb 3(\omega)$ is identically zero
in an open neighborhood of $\omega\sb\ast$.
Applying $(\eub{J}\eub{L})^2$ to \eqref{dvk-jh2-e2-e4},
we get:
\begin{equation}\label{dvk-jh4-e3-e4}
(\eub{J}\eub{L})^3
{\bm e}\sb 3(\omega)
=
\sigma\sb 3(\omega)(\eub{J}\eub{L})^2{\bm e}\sb 3(\omega)
+
\sigma\sb 4(\omega)(\eub{J}\eub{L})^2{\bm e}\sb 4(\omega).
\end{equation}
Coupling with $\eub{J}^{-1}{\bm e}\sb 4$
and using
$
\langle
\eub{J}^{-1}{\bm e}\sb 4,(\eub{J}\eub{L})^3{\bm e}\sb 3\rangle
=\langle
{\bm e}\sb 3,\eub{L} \eub{J} \eub{L}{\bm e}\sb 3\rangle=0
$
and
$\langle
\eub{J}^{-1}{\bm e}\sb 4,(\eub{J}\eub{L})^2{\bm e}\sb 4\rangle
=-\langle
{\bm e}\sb 4,\eub{L} \eub{J}\eub{L}{\bm e}\sb 4\rangle=0
$
(due to anti-selfadjointness of $\eub{L}\eub{J}\eub{L}$),
we have
\begin{equation}\label{dvk-s3asdf}
\sigma\sb 3(\omega)
\langle \eub{J}^{-1}{\bm e}\sb 4,
(\eub{J}\eub{L})^2{\bm e}\sb 3\rangle=0.
\end{equation}
The factor at $\sigma\sb 3(\omega)$
is nonzero.
Indeed, using \eqref{dvk-sigma0100},
\[
\langle
\eub{J}^{-1}{\bm e}\sb 4,(\eub{J}\eub{L})^2{\bm e}\sb 3\rangle
\at{\omega\sb\ast}
=
\langle
\eub{J}^{-1}{\bm e}\sb 4,\,
\sigma\sb 2{\bm e}\sb 1
+
\sigma\sb 3 \eub{J}\eub{L}{\bm e}\sb 3
+
\sigma\sb 4{\bm e}\sb 3
\rangle
\at{\omega\sb\ast}
=
\langle
\eub{J}^{-1}{\bm e}\sb 4,
{\bm e}\sb 1
\rangle
\at{\omega\sb\ast}
=
-\langle
{\bm e}\sb 4,\varPhi
\rangle
\at{\omega\sb\ast},
\]
which is nonzero due to \eqref{dvk-e4-phi-nonzero}.
We conclude that
$\sigma\sb 3(\omega)$ is identically zero
in an open neighborhood of $\omega\sb\ast$.
We take $\mathcal{O}$ small enough so that
$\sigma\sb 3\at{\mathcal{O}}\equiv 0$.

Near $\omega=\omega\sb\ast$,
the eigenvalues of $\eub{J}\eub{L}$
which are located inside a small contour around $\lambda=0$
coincide with the eigenvalues of the matrix $M$,
defined in \eqref{dvk-rbm}.
Since $\sigma\sb 3$ is identically zero
in $\mathcal{O}$,
these eigenvalues satisfy
\begin{equation}
\lambda^2(\lambda^2-\sigma\sb 4(\omega))=0,
\qquad
\omega\in\mathcal{O}.
\end{equation}
By \eqref{dvk-sigma4-sigma2},
if $\p\sb\omega Q(\phi)$
changes sign at $\omega\sb\ast$,
so does $\sigma\sb 4(\omega)$,
hence
in a one-sided open neighborhood of $\omega\sb\ast$
there are two real eigenvalues of $\eub{J}\eub{L}$,
one positive
(indicating the linear instability)
and one negative.
\end{proof}

\begin{remark}
This argument
is slightly longer than
a similar computation in \cite{MR1995870}
since we do not assume that
$\phi$ could be chosen purely real
(allowing for a common ansatz used in \cite{MR1344729}
in the context of the nonlinear Dirac equation),
and consequently
we could not
take ${\bm e}\sb j$ to be ``imaginary''
for $j$ odd
and ``real''
for $j$ even,
and
enjoy the vanishing of
$\langle \eub{J}^{-1}{\bm e}\sb j,{\bm e}\sb k\rangle$
for $j+k$ is even.
\end{remark}

\section{Concluding remarks}

In the conclusion, let us make several observations.

\begin{remark}
In the case of the nonlinear Schr\"odinger equation,
the function $\mu(\omega)$ is strictly positive
(this is due to positive-definiteness
of $\eur{L}\sb{-}$ in \eqref{dvk-def-l0-l1},
which results in positivity of
\eqref{dvk-mu-nonzero};
see~\cite{VaKo,MR1995870}),
so that the collisions of eigenvalues at
$\lambda=0$ always happen according to the following scenario:
when $dQ/d\omega$ changes from negative to positive
(no matter whether this happens as $\omega$ increases or decreases),
there is a pair of eigenvalues
on the imaginary axis
colliding at $\lambda=0$
and proceeding along the real axis.
Then, further, 
if $dQ/d\omega$ changes from positive to negative,
this pair of real eigenvalues
return to $\lambda=0$
and then retreat onto the imaginary axis.
For the Dirac equation, we can not rule out
that $\mu(\omega)$ changes the sign (becoming negative).
If this were the case,
vanishing of $dQ/d\omega$
would be accompanied with the reversed bifurcation mechanism:
as $dQ/d\omega$ goes from positive to negative,
another pair of imaginary eigenvalues
collide at $\lambda=0$
and proceed along the real axis.
We do not have examples of particular nonlinearities
which lead to such a scenario.
\end{remark}

\begin{remark}
In one dimension,
for the nonlinearity $g(s)=1-s^k$,
$k\in\N$,
using the asymptotics as in \cite{2008arXiv0812.2273G},
one can derive that,
as $\omega\to 1-$,
the function
$\mu(\omega)$
defined in \eqref{dvk-def-mu}
has the asymptotics
\[
\mu(\omega)=O((1-\omega^2)^{\frac{1}{k}-\frac{7}{2}}).
\]
There is the same asymptotics
in the case of the nonlinear Schr\"odinger equation
\eqref{dvk-nls}
with the same nonlinearity $g(s)$.
In particular,
$\lim\limits\sb{\omega\to 1-}\mu(\omega)=+\infty$,
hence $\mu(\omega)$
remains positive
for $\omega$ sufficiently close to $1$
(precisely as for the nonlinear Schr\"odinger equation,
when $\mu(\omega)>0$ for all $\omega$),
suggesting that
the bifurcation scenario for the nonlinear Dirac equation
near $\omega=1$
is the same as for the nonlinear Schr\"odinger equation.
\end{remark}

\begin{remark}
For the nonlinear Dirac equation
in one dimension
with the nonlinearity $g(s)=1-s$
(``Soler model''),
the solitary waves exist
for $\omega\in(0,1)$.
One can explicitly compute that
the charge is given by
$Q(\omega)=\int\sb{\R}\abs{\psi}^2\,dx
=\frac{2\sqrt{1-\omega^2}}{\omega}$
(see e.g. \cite{PhysRevD.12.3880})
so that
$dQ/d\omega<0$ for $\omega\in(0,1)$,
implying that there are no eigenvalues
of $\eub{J}\eub{L}$
colliding
at $\lambda=0$
for any $\omega\in(0,1)$.
According to our numerical results
\cite{dirac-1d-arxiv},
we expect that in this model
the solitary waves are spectrally stable,
at least for $\omega$ sufficiently close to $1$.
\end{remark}

\begin{remark}
Even if $dQ/d\omega$ never vanishes,
so that there are no bifurcations
of nonzero real eigenvalues from $\lambda=0$,
there may be nonzero real eigenvalues
present in the spectrum of {\it all} solitary waves.
We expect that the Vakhitov-Kolokolov criterion
could again be useful here
when applied in the nonrelativistic limit
$\omega\to m$,
when the properties of the nonlinear Dirac equation
are similar to properties of the nonlinear Schr\"odinger
equation.
This idea has been mentioned in \cite{PhysRevE.82.036604}.
Our preliminary results
indicate that
in one dimension,
for $\omega$ sufficiently close to $1$,
the nonlinearity with $g(s)=1-s^k+o(s^k)$
with $k=1$ and $k=2$
does not produce nonzero real eigenvalues,
while for $k\ge 3$ there are two real eigenvalues,
one positive (leading to linear instability) and one negative.
\end{remark}

\notyet{
\section{Conclusion}

We have shown that
the Vakhitov-Kolokolov stability criterion
enters the Grillakis-Shatah-Strauss
formalism is still meaningful for the
nonlinear Dirac equation and similar
equations.
Namely,
the point when the Vakhitov-Kolokolov
stability condition breaks down
is a bifurcation point;
beyond this point, an eigenvalue
with positive real part
emerges from $\lambda=0$.

Let us emphasize, though,
that
while the breaking down of
the Vakhitov-Kolokolov stability condition
definitely indicates the spectral instability,
we can make no conclusion about stability
properties of solitary waves
in the region where this condition is satisfied.
This is in contrast with the nonlinear Schr\"odinger,
Klein-Gordon, and similar systems,
where the Vakhitov-Kolokolov condition
is sufficient to conclude that the solitary waves
minimize the energy under the charge constraint
and as a consequence are orbitally stable.

}


\bigskip
\noindent
ACKNOWLEDGMENTS.
The author is grateful
for fruitful discussions
and helpful comments
to
Gregory Berkolaiko,
Nabile Boussaid,
Marina Chugunova,
Scipio Cuccagna,
Maria Esteban,
Linh Viet Nguyen,
Todd Kapitula,
Ruomeng Lan,
Dmitry Pelinovsky,
Iosif Polterovich,
Bjorn Sandstede,
Eric S\'er\'e,
Walter Strauss,
Wilhelm Schlag,
Boris Vainberg,
and Michael Weinstein.


\def\cprime{$'$} \def\cprime{$'$} \def\cprime{$'$} \def\cprime{$'$}
  \def\cprime{$'$} \def\cprime{$'$} \def\cprime{$'$} \def\cprime{$'$}
  \def\cprime{$'$} \def\cydot{\leavevmode\raise.4ex\hbox{.}} \def\cprime{$'$}
  \def\cydot{\leavevmode\raise.4ex\hbox{.}} \def\cprime{$'$} \def\cprime{$'$}
  \def\cprime{$'$} \def\cprime{$'$}

\end{document}